\documentclass[12pt]{article}
\usepackage{amsmath,amsthm,amsfonts,amssymb}
\usepackage{fullpage}
\usepackage{multirow}
\usepackage{array}
\usepackage{bbm}
\usepackage[noblocks]{authblk}
\usepackage{verbatim}
\usepackage{tikz}
\usepackage{float}
\usepackage{enumerate}
\usepackage{diagbox}
\usepackage{soul}
\usepackage{url}
\usepackage{mathtools}
\usepackage{hyperref}
\usepackage{listings}

\newtheorem{theorem}{Theorem}[section]

\newtheorem{lemma}[theorem]{Lemma}

\newtheorem{corollary}[theorem]{Corollary}
\newtheorem{conjecture}[theorem]{Conjecture}

\theoremstyle{definition}

\newlength{\Oldarrayrulewidth}

\begin{document}

\title{Arithmetic Progressions of Integers that are\\Relatively Prime to their Digital Sums}
\author[1]{Ryan~Blau\thanks{ryan.blau@yotes.collegeofidaho.edu}}
\author[2]{Joshua~Harrington\thanks{joshua.harrington@cedarcrest.edu}}
\author[3]{Sarah~Lohrey\thanks{slohrey@brynmawr.edu}}
\author[4]{Eliel~Sosis\thanks{esosis@umich.edu}}
\author[5]{Tony~W.~H.~Wong\thanks{wong@kutztown.edu}}
\affil[1]{Department of Mathematics \& Physical Sciences, The College of Idaho}
\affil[2]{Department of Mathematics, Cedar Crest College}
\affil[3]{Department of Mathematics, Bryn Mawr College}
\affil[4]{Department of Mathematics, University of Michigan}
\affil[5]{Department of Mathematics, Kutztown University of Pennsylvania}
\date{\today}

\maketitle

\begin{abstract}

For an integer $b\geq 2$, we call a positive integer $b$-anti-Niven if it is relatively prime to the sum of the digits in its base-$b$ representation.  In this article, we investigate the maximum lengths of arithmetic progressions of $b$-anti-Niven numbers.
\end{abstract}

\section{Introduction}

Throughout this paper, let $b\geq2$ be an integer. For all positive integers $n$, let $s_b(n)$ denote the sum of the digits in the base-$b$ expansion of $n$, i.e., if $n=\sum_{j=0}^m a_j b^j$, where $m$ is a nonnegative integer and $0\leq a_j \leq b-1$ are integers for each $0\leq j\leq m$, then $s_b(n)=\sum_{j=0}^m a_j$.

For positive integers $n$, $d$, and $t$, we call the sequence $\{n+jd:0\leq j\leq t-1\}$ a \emph{$d$-AP of length $t$} and we call the sequence $\{n+jd:j\geq 0\}$ a \emph{$d$-AP of infinite length}. A positive integer $n$ is $b$-\emph{Niven} if $s_b(n)\mid n$.  If every term of a $d$-AP is $b$-Niven, we call it a \emph{$b$-Niven $d$-AP}.  We note that a $b$-Niven $1$-AP is a sequence of consecutive Niven numbers.  

In 1993, Cooper and Kennedy \cite{ck1} showed that the maximum length of a $10$-Niven $1$-AP is $20$.  Grundman \cite{g} generalized this result in 1994 by showing that the maximum length of a $b$-Niven $1$-AP is $2b$.  These maximum lengths were shown to be attainable by Wilson \cite{w}.  More recently, Grundman, Harrington, and Wong \cite{ghw} investigated maximum length $b$-Niven $d$-APs for $d>1$ and Harrington, Litman, and Wong \cite{hlw} showed that every infinite $d$-AP contains infinitely many $b$-Niven numbers.

In 1975, Olivier \cite{olivier} studied sets $S_b=\{n\in\mathbb{Z}:\gcd(n,s_b(n))=1\}$ and showed that the natural density of these sets is $\frac{6}{\pi^2}\prod_{p\mid(b-1)}\frac{p}{p+1}$.  In 1997, Cooper and Kennedy \cite{ck2} published a weaker result that established Olivier's density as an upper bound for the density of $S_{10}$.

In this paper, we define a positive integer $n$ to be \emph{$b$-anti-Niven} if $\gcd(s_b(n),n)=1$.  If every term of a $d$-AP is $b$-anti-Niven, then we call it a \emph{$b$-anti-Niven $d$-AP}.  In Section~\ref{sec:general} we give necessary and sufficient conditions on $d$, $b$, and $n$ for which the $d$-AP $\{n+jd:j\geq 0\}$ contains at least one $b$-anti-Niven number.  We also show that there is no $b$-anti-Niven $d$-AP of infinite length, but for any $b$ and $t$, there are infinitely many $b$-anti-Niven $d$-APs of length $t$.  In Section~\ref{sec:dAP} we investigate the maximum length of $b$-anti-Niven $d$-APs when $b$ and $d$ satisfy various constraints.

\section{$b$-anti-Niven Numbers in $d$-APs}\label{sec:general}
In this section, we are going to give several general results on $b$-anti-Niven numbers in $d$-APs. 

\begin{lemma}\label{lem:b-1}
Let $\delta$ be a positive integer such that $\delta\mid(b-1)$. Then for all positive integers $n$, $\delta\mid n$ if and only if $\delta\mid s_b(n)$.
\end{lemma}
\begin{proof}
Let $n=\sum_{j=0}^m a_j b^j$, where $m$ is a nonnegative integer and $0\leq a_j \leq b-1$ are integers for each $0\leq j\leq m$. The proof follows from the simple observation that $b\equiv1\pmod{\delta}$ and thus $\sum_{j=0}^m a_j b^j\equiv\sum_{j=0}^m a_j\pmod{\delta}$.
\end{proof}

\begin{theorem}
The $d$-AP of infinite length $\{n+jd:j\geq0\}$ contains a $b$-anti-Niven number if and only if $\gcd(n,d,b-1)=1$.
\end{theorem}
\begin{proof}
Assuming that $\gcd(n,d,b-1)=1$, we have $\gcd(s_b(n),s_b(d),b-1)=1$ by Lemma~\ref{lem:b-1}. By Proposition~2.6 of Harrington, Litman, and Wong \cite{hlw}, there exists a positive multiple $\overline{d}$ of $d$ such that $\gcd(s_b(n),s_b(\overline{d}))=1$. Let $k$ be a positive integer such that $s_b(n)+k\cdot s_b(\overline{d})=p$ is a prime with $p>\max(b,\overline{d})$, and we further let $m_0=\lfloor\log_b(n)\rfloor+1$ and $m_i=m_{i-1}+\lfloor\log_b(\overline{d})\rfloor+1$ for all $1\leq i\leq k$. Consider $n+j\overline{d}$ and $n+j'\overline{d}$, where $j=\sum_{i=0}^kb^{m_i}$ and $j'=j-b^{m_k}+b^{m_k+1}$. Note that both $s_b(n+j\overline{d})$ and $s_b(n+j'\overline{d})$ are equal to $s_b(n)+k\cdot s_b(\overline{d})=p$, and $(n+j'\overline{d})-(n+j\overline{d})=b^{m_k}(b-1)d$ is not divisible by $p$ since $p>\max(b,\overline{d})$. Hence, at least one of $n+j\overline{d}$ and $n+j'\overline{d}$ is our desired $b$-anti-Niven number in the given $d$-AP.


Conversely, if $\gcd(n,d,b-1)=\delta>1$, then $\delta\mid\gcd(n+jd,s_b(n+jd))$ for all integers $j\geq0$ by Lemma~\ref{lem:b-1}. Therefore, $\{n+jd:j\geq0\}$ does not contain any $b$-anti-Niven numbers.
%
\end{proof}

The following theorem is a consequence of a result of Harrington, Litman, and Wong \cite{hlw} who showed that every arithmetic progression of infinite length contains at least one $b$-Niven number $n$ with $s_b(n)\neq 1$.
\begin{theorem}
For any positive integer $d$, there is no $b$-anti-Niven $d$-AP of infinite length.
\end{theorem}

Our next theorem shows that there exist arithmetic progressions of arbitrary length containing only $b$-anti-Niven numbers.

\begin{theorem}\label{thm:arbitrary}
For every positive integer $t$, there exist positive integers $n$ and $d$ such that $\{n+jd:0\leq j\leq t-1\}$ is a $b$-anti-Niven $d$-AP of length $t$.
\end{theorem}
\begin{proof}
Let $m$ be a positive integer such that $b^m/(m(b-1)+1)\geq t$, and let $d=b(b^m-1)(m(b-1)+1)$. Consider the $d$-AP $\{(d+1)+jd:0\leq j\leq t-1\}$. For all $0\leq j\leq t-1$, note that $\widetilde{j}=(j+1)(m(b-1)+1)\leq b^m$. Hence,
\begin{align*}
s_b((d+1)+jd)&=s_b(\widetilde{j}(b^{m+1}-b)+1)\\
&=s_b\big((\widetilde{j}-1)b^{m+1}+b(b^m-1-(\widetilde{j}-1))+1\big)\\
&=s_b(\widetilde{j}-1)+s_b\big(b^m-1-(\widetilde{j}-1)\big)+1\\
&=s_b(\widetilde{j}-1)+s_b\left(\sum_{j=0}^{m-1}(b-1)b^j-(\widetilde{j}-1)\right)+1\\
&=s_b(\widetilde{j}-1)+m(b-1)-s_b(\widetilde{j}-1)+1\\
&=m(b-1)+1.
\end{align*}
Since $(d+1)+jd\equiv1\pmod{m(b-1)+1}$, we conclude that $(d+1)+jd$ is $b$-anti-Niven for all $0\leq j\leq t-1$.
\end{proof}

Although Theorem~\ref{thm:arbitrary} shows that there are $b$-anti-Niven $d$-APs of arbitrary length, the maximum length of a $b$-anti-Niven $d$-AP is bounded above by $b-2$ for many values of $b$ and $d$, as shown in the following theorem.

\begin{theorem}\label{thm:bound}
For $b>2$ and a positive integer $d$, let $p$ be the smallest prime such that $p\mid(b-1)$ and $p\nmid d$.  Then every $b$-anti-Niven $d$-AP has length at most $p-1$.
\end{theorem}
\begin{proof}
Since $p\nmid d$, every $d$-AP of length $p$ contains a multiple of $p$. By Lemma~\ref{lem:b-1}, this multiple of $p$ is not $b$-anti-Niven. Hence, the maximum length of a $d$-AP that contains only $b$-anti-Niven numbers is at most $p-1$.
\end{proof}

\section{Maximum Length $b$-anti-Niven $d$-APs}\label{sec:dAP}
Theorem~\ref{thm:bound} in the previous section gives a bound on the maximum length of certain $b$-anti-Niven $d$-APs. We begin this section by demonstrating that there are instances when this bound is achieved.  Theorems~\ref{thm:b>2,1AP} and \ref{thm:2AP} investigate $1$-APs, i.e. sequences of consecutive $b$-anti-Niven numbers,  and $2$-APs, respectively. The following lemma will be a common tool in establishing these two theorems.

%
%

\begin{lemma}\label{lem:b^m=b}
For all finite collections of distinct primes $q_1,q_2,\dotsc,q_t$, there exist infinitely many positive integers $m$ such that $b^m\equiv b\pmod{q}$.
\end{lemma}
\begin{proof}
Without loss of generality, assume that there exists $0\leq t'\leq t$ such that $q_i\nmid b$ for all $1\leq i\leq t'$ and $q_i\mid b$ for all $t'+1\leq i\leq t$. By Euler's theorem, $b^{k\varphi(q_1q_2\dotsb q_{t'})}\equiv1\pmod{q_1q_2\dotsb q_{t'}}$ for every positive integer $k$. Hence, $m=k\varphi(q_1q_2\dotsb q_{t'})+1$ is our desired choice of integer.
\end{proof}


\begin{theorem}\label{thm:b>2,1AP}
For $b>2$, let $p$ be the smallest prime such that $p\mid(b-1)$. Then the maximum length of a sequence of consecutive $b$-anti-Niven numbers is $p-1$. Furthermore, there exist infinitely many such sequences of length $p-1$.
\end{theorem}
\begin{proof}
By Theorem~\ref{thm:bound}, the maximum length of a $b$-anti-Niven $1$-AP is at most $p-1$. It remains to show that such sequences occur infinitely often. Let $q_1,q_2,\dotsc,q_t$ be all primes less than $p$. By Lemma~\ref{lem:b^m=b}, there exist infinitely many positive integers $m$ such that $b^m\equiv b\pmod{q_1q_2\dotsb q_t}$. Now, for all $0\leq j\leq p-2$, we have $s_b(b^m+j)=j+1$.  Since $j+1<p$, for any prime divisor $q$ of $j+1$, we have $b^m+j\equiv b+j\equiv b-1\not\equiv 0\pmod{q}$. Therefore, $\gcd(b^m+j,s_b(b^m+j))=1$, implying that $\{b^m+j:0\leq j\leq p-2\}$ forms a sequence of $p-1$ consecutive $b$-anti-Niven numbers.
\end{proof}

\begin{theorem}\label{thm:2AP}
Let $b>2$ be such that $b\neq2^r+1$ for any integer $r$, and let $p$ be the smallest odd prime such that $p\mid(b-1)$. Then the maximum length of a $b$-anti-Niven $2$-AP is $p-1$. Furthermore, there exist infinitely many such sequences of length $p-1$.
\end{theorem}
\begin{proof}
By Theorem~\ref{thm:bound}, the maximum length of a $b$-anti-Niven $2$-AP is at most $p-1$. It remains to show that such sequences occur infinitely often. Let $q_1,q_2,\dotsc,q_t$ be all primes less than or equal to $b$. By Lemma~\ref{lem:b^m=b}, there exist infinitely many positive integers $m$ such that $b^m\equiv b\pmod{q_1q_2\dotsb q_t}$.

Consider the case when $b$ is even. For all $0\leq j\leq p-2$, we have $s_b(b^m+2j+1)=2(j+1)$. Since $j+1<p$, for any prime divisor $q$ of $2(j+1)$, we have $b^m+2j+1\equiv b+2j+1\equiv b-1\not\equiv0\pmod{q}$. Therefore, $\gcd(b^m+2j+1,s_b(b^m+2j+1))=1$, implying that $\{b^m+2j+1:0\leq j\leq p-2\}$ forms a $b$-anti-Niven $2$-AP of length $p-1$.

Next, consider the case when $b$ is odd. For all $0\leq j\leq(p-1)/2$, we have $s_b(b^m+b-p+2j)=1+b-p+2j$. Since $1+b-p+2j\leq b$, for any prime divisor $q$ of $1+b-p+2j$, we have $b^m+b-p+2j\equiv 2b-p+2j\equiv 2b-p+2j-2(1+b-p+2j)\equiv p-2j-2\pmod{q}$. Note that $-1\leq p-2j-2\leq p-2$, so none of these odd numbers share a common prime factor with $b-1$. Hence, $\gcd(p-2j-2,1+b-p+2j)=\gcd(p-2j-2,b-1)=1$, implying that $p-2j-2\not\equiv0\pmod{q}$. Thus, $\gcd(b^m+b-p+2j,s_b(b^m+b-p+2j))=1$ when $0\leq j\leq (p-1)/2$.

Furthermore, for all $0\leq j\leq(p-5)/2$, we have $s_b(b^m+b+1+2j)=3+2j$. Since $3+2j\leq p-2<b$, for any prime divisor $q$ of $3+2j$, we have $b^m+b+1+2j\equiv2b+1+2j\equiv2b+1+2j-(3+2j)\equiv2(b-1)\pmod{q}$. Note that $2(b-1)\not\equiv0\pmod{q}$ since $q$ is an odd prime less than $p$. Thus, $\gcd(b^m+b+1+2j,s_b(b^m+b+1+2j))=1$ when $0\leq j\leq(p-5)/2$. Therefore, $\{b^m+b-p+2j:0\leq j\leq(p-1)/2\}\cup\{b^m+b+1+2j:0\leq j\leq(p-5)/2\}$ forms a $b$-anti-Niven $2$-AP of length $p-1$.
\end{proof}

So far, the theorems in this section have shown that the bound provided in Theorem~\ref{thm:bound} is attainable. However, there are infinitely many instances when the maximum length of $b$-anti-Niven $d$-APs does not attain this bound. The following theorem illustrates one such instance.

\begin{theorem}\label{thm:smallerthanbound}
Let $b\geq6$ be even, and let $3\leq d\leq b/2$ be an odd integer. Then the maximum length of a $b$-anti-Niven $d$-AP is at most $\lceil2b/d\rceil+2$.
\end{theorem}
Note that when $b-1$ is an odd prime, the bound given by Theorem~\ref{thm:bound} is $b-2$, while the bound given by Theorem~\ref{thm:smallerthanbound} is strictly smaller when $b>15$.
\begin{proof}[Proof of Theorem~\ref{thm:smallerthanbound}]
Suppose there are two $b$-anti-Niven numbers from a $d$-AP in the interval $[ab,ab+b-1]$. Let these two numbers be $ab+a_0$ and $ab+a_0+d$ for some nonnegative integers $a$ and $a_0\leq b-1-d$. Recalling that $d$ is odd, there exists $\chi\in\{0,1\}$ so that $a_0+\chi d$ is even. As a result, $ab+a_0+\chi d$ is also even since $b$ is even. Since $ab+a_0+\chi d$ is $b$-anti-Niven, $s_b(ab+a_0+\chi d)=s_b(a)+a_0+\chi d$ is odd, implying that $s_b(a)$ is odd.

Note that there are at most $\lceil2b/d\rceil$ terms from a $d$-AP in the interval $[ab,(a+1)b+(b-1)]$. Hence, if there is a $b$-anti-Niven $d$-AP of length $\lceil2b/d\rceil+3$, then there exists a nonnegative integer $a$ such that each of the intervals $[ab,ab+b-1]$, $[(a+1)b,(a+1)b+b-1]$, and $[(a+2)b,(a+2)b+b-1]$ contains at least two terms from this $d$-AP. From the above observation, we conclude that $s_b(a)$, $s_b(a+1)$, and $s_b(a+2)$ are all odd, which is a contradiction. This establishes an upper bound for the maximum length of a $b$-anti-Niven $d$-AP as stated in the theorem.
\end{proof}

We now turn our attention to $d$-APs for which Theorem~\ref{thm:bound} does not apply.

\begin{theorem}\label{thm:bevenb-1AP}
Let $b$ be even. Then the maximum length of a $b$-anti-Niven $(b-1)$-AP is $2b+1$. Furthermore, there exist infinitely many such sequences of length $2b+1$.
\end{theorem}
\begin{proof}
Suppose there is a $b$-anti-Niven $(b-1)$-AP $\mathcal{S}$ of length at least $2b+1$. Since $\gcd(b-1,b)=1$, there exist two terms in $\mathcal{S}$ that are multiples of $b$. Let these two terms be $ab$ and $ab+(b-1)b$ for some positive integer $a$. Note that $ab$ is even, so $s_b(ab)=s_b(a)$ is odd. Since $ab+2(b-1)=(a+1)b+b-2$ is an even term in $\mathcal{S}$, the digit sum $s_b((a+1)b+b-2)=s_b(a+1)+b-2$ must be odd, implying that $s_b(a+1)$ is also odd. Hence, $a=cb+b-1$ for some nonnegative integer $c$, where $s_b(c)$ is even. In other words, $ab=cb^2+(b-1)b$ and $ab+(b-1)b=(c+1)b^2+(b-2)b$. Also, $s_b(c+1)$ is odd since $s_b((c+1)b^2+(b-2)b)=s_b(c+1)+b-2$ is odd.

Now, note that $cb^2$ and $s_b(cb^2)=s_b(c)$ are even, so $cb^2$ is not in $\mathcal{S}$. Similarly, $(c+1)b^2+(b-1)b+b-2$ and $s_b((c+1)b^2+(b-1)b+b-2)=s_b(c+1)+b-1+b-2$ are even, so $(c+1)b^2+(b-1)b+b-2=cb^2+(2b+2)(b-1)$ is also not in $\mathcal{S}$. Therefore, $\mathcal{S}$ is a subsequence of $\{cb^2+j(b-1):1\leq j\leq2b+1\}$, thus the maximum length of a $b$-anti-Niven $(b-1)$-AP is at most is $2b+1$.

It remains to show that such sequences occur infinitely often. Let $c$ be a nonnegative integer such that $s_b(c+1)=s_b(c)+1$. Then it is not difficult to observe that $s_b(cb^2+j(b-1))=s_b(c)+b-1$ for $1\leq j\leq b$ and $b+2\leq j\leq 2b$, and $s_b(cb^2+j(b-1))=s_b(c)+2(b-1)$ for $j\in\{b+1,2b+1\}$. Hence, it suffices to show that there exist infinitely many positive integers $c$ such that
\begin{itemize}
\item $b\nmid(c+1)$,
\item $\gcd(cb^2+j(b-1),s_b(c)+b-1)=1$ for $1\leq j\leq b$ and $b+2\leq j\leq 2b$, and
\item $\gcd(cb^2+j(b-1),s_b(c)+2(b-1))=1$ for $j\in\{b+1,2b+1\}$.
\end{itemize}
Let $p_1,p_2,\dotsc,p_t$ be all primes less than or equal to $2b$. By Lemma~\ref{lem:b^m=b}, there exist infinitely many positive integers $m$ such that $b^{m+1}\equiv b\pmod{p_1p_2\dotsb p_t}$. In other words, $p_1p_2\dotsb p_t\mid b(b^m-1)$. Let $P=b^m+1$. Since $\gcd(b^m+1,b)=1$ and $\gcd(b^m+1,b^m-1)=1$, we have $\gcd(P,p_1p_2\dotsb p_t)=1$. Next, consider $q_1,q_2,\dotsc,q_\tau$ be all prime factors of $b^{m-1}+1$. Hence, $b^{m-1}\equiv-1\pmod{q_1q_2\dotsb q_\tau}$. Let $r_1,r_2,\dotsc,r_{(P-b+1)/2}$ be positive integers, where $r_{i+1}-r_i\geq m+1$ for all $1\leq i\leq(P-b-1)/2$, be defined as follows.
\begin{itemize}
\item If $(P-b+1)/2$ is odd, then
$$b^{r_i+2}\equiv\begin{cases}
-1\pmod{q_1q_2\dotsb q_\tau}&\text{if }1\leq i\leq(P-b-1)/4;\\
1\pmod{q_1q_2\dotsb q_\tau}&\text{otherwise}.
\end{cases}$$
\item If $(P-b+1)/2$ is even, then
$$b^{r_i+2}\equiv\begin{cases}
-1\pmod{q_1q_2\dotsb q_\tau}&\text{if }1\leq i\leq(P-b-3)/4;\\
1\pmod{q_1q_2\dotsb q_\tau}&\text{if }(P-b+1)/4\leq i\leq(P-b-3)/2;\\
b\pmod{q_1q_2\dotsb q_\tau}&\text{otherwise}.
\end{cases}$$
\end{itemize}
Now, let $c=\sum_{i=1}^{(P-b+1)/2}b^{r_i}(b^m+1)$. Since $b\mid c$, we have $b\nmid(c+1)$. Next, $s_b(c)=P-b+1$ from our construction, thus $s_b(c)+b-1=P=b^m+1$, which is a factor of $c$. Recalling that $P$ is relatively prime to all positive integers up to $2b$, we have $\gcd(cb^2+j(b-1),P)=1$ for all $1\leq j\leq 2b$. It remains to prove that $\gcd(cb^2+j(b-1),s_b(c)+2(b-1))=1$ for $j\in\{b+1,2b+1\}$. Note that $s_b(c)+2(b-1)=P-b+1+2(b-1)=P+b-1=b(b^{m-1}+1)$. For any prime factor $q$ of $b^{m-1}+1$, we clearly have $q\nmid b$. Moreover, $q\nmid(b-1)$, or otherwise, $q\mid b(b^{m-1}+1)$ and $q\mid(b-1)$ imply $q\mid P$, contradicting that $P$ is relatively prime to all positive integers up to $2b$. If $(P-b+1)/2$ is odd, then
\begin{align*}
    cb^2+(b+1)(b-1)&=\left(\sum_{i=1}^{(P-b+1)/2}b^{r_i+2}\right)P+b^2-1\\
    &\equiv P+b^2-1\\
    &\equiv P+b^2-1-(P+b-1)\\
    &\equiv b(b-1)\\
    &\not\equiv0\pmod{q}
\end{align*}
and $cb^2+(2b+1)(b-1)=cb^2+(b+1)(b-1)+b(b-1)\equiv2b(b-1)\not\equiv0\pmod{q}$. If $(P-b+1)/2$ is even, then 
\begin{align*}
    cb^2+(b+1)(b-1)
    &=\left(\sum_{i=1}^{(P-b+1)/2}b^{r_i+2}\right)P+b^2-1\\
    &\equiv 2bP+b^2-1\\
    &\equiv 2bP+b^2-1-2b(P+b-1)\\
    &\equiv-(b-1)^2\\
    &\not\equiv0\pmod{q}
\end{align*}
and $cb^2+(2b+1)(b-1)=cb^2+(b+1)(b-1)+b(b-1)\equiv-(b-1)^2+b(b-1)\equiv b-1\not\equiv0\pmod{q}$. Finally, our proof is completed by noticing that $\gcd(cb^2+j(b-1),b)=1$ for $j\in\{b+1,2b+1\}$.
\end{proof}

To complete the investigation on $1$-APs, we provide the following corollary by choosing $b=2$ in Theorem~\ref{thm:bevenb-1AP}.

\begin{corollary}\label{cor:b=2,1AP}
For $b=2$, the maximum length of a sequence of consecutive $2$-anti-Niven numbers is $5$. Furthermore, there exist infinitely many such sequences of length $5$.
\end{corollary}

\section{Concluding Remarks}\label{sec:conclusion}

Theorem~\ref{thm:2AP} establishes the that the maximum length of a $b$-anti-Niven $2$-AP is at most $b-2$ when $b-1$ has an odd prime divisor.  Although we have not established an upper bound for the maximum length of a $b$-anti-Niven $2$-AP when $b=2^r+1$ for some nonnegative integer $r$, the following theorem establishes a lower bound.

\begin{theorem}\label{thm:c13}
Let $b=2^r+1$ for some nonnegative integer $r$. Then the maximum length of a $b$-anti-Niven $2$-AP is at least $b$.
\end{theorem}
\begin{proof}
If $r=0$, then $b=2$, and $\{2,4\}$ forms a $2$-anti-Niven $2$-AP of length $2$. If $r>0$, then $b$ is odd. For all $0\leq j\leq(b-1)/2$, we have $\gcd(b+2j,s_b(b+2j))=\gcd(b+2j,1+2j)=\gcd(b+2j,b-1)=\gcd(2^r+1+2j,2^r)=1$. Furthermore, for all $0\leq j\leq(b-3)/2$, we have $\gcd(2b+1+2j,s_b(2b+1+2j))=\gcd(2b+1+2j,3+2j)=\gcd(2b+1+2j,2b-2)=\gcd(2^{r+1}+3+2j,2^{r+1})=1$. Therefore, $\{b+2j:0\leq j\leq(b-1)/2\}\cup\{2b+1+2j:0\leq j\leq(b-3)/2\}$ forms a $b$-anti-Niven $2$-AP of length $b$.
\end{proof}

Theorem~\ref{thm:bevenb-1AP} establishes the that the maximum length of a $b$-anti-Niven $(b-1)$-AP is at most $2b+1$ when $b$ is even.  The following theorem establishes a lower bound when $b$ is an odd prime.

\begin{theorem}\label{thm:c14}
Let $b$ be an odd prime. Then the maximum length of a $b$-anti-Niven $(b-1)$-AP is at least $2b+1$.
\end{theorem}
\begin{proof}
Clearly, $1$, $b$, and $b^2$ are $b$-anti-Niven. For all $1\leq j\leq b-1$, we have
\begin{align*}
\gcd\big(b+j(b-1),s_b(b+j(b-1))\big)&=\gcd\big((j+1)b-j,s_b(jb+b-j)\big)\\
&=\gcd((j+1)b-j,j+b-j)\\
&=\gcd((j+1)b-j,b)=1
\end{align*}
since $b$ is a prime. Furthermore, for all $1\leq j\leq b-1$, we have
\begin{align*}
\gcd\big(b^2+j(b-1),s_b(b^2+j(b-1))\big)&=\gcd\big(b^2+jb-j,s_b(b^2+(j-1)b+b-j)\big)\\
&=\gcd(b^2+jb-j,1+j-1+b-j)\\
&=\gcd(b^2+jb-j,b)=1.
\end{align*}
Therefore, $\{1,b\}\cup\{b+j(b-1):1\leq j\leq b-1\}\cup\{b^2\}\cup\{b^2+j(b-1):1\leq j\leq b-1\}$ forms a $b$-anti-Niven $(b-1)$-AP of length $2b+1$.
\end{proof}

Recall that Theorem~\ref{thm:smallerthanbound} shows that the upper bound on the maximum length of a $b$-anti-Niven $d$-AP given by Theorem~\ref{thm:bound} may not be achievable for even $b$. However, when $b$ is odd, computational data suggests otherwise. Of course, if $d$ is odd, then Theorem~\ref{thm:bound} implies that the maximum length of a $b$-anti-Niven $d$-AP is at most $1$, which is clearly attainable. It is more interesting to investigate if $d$ is even.  The next conjecture addresses this more interesting case and generalizes Theorem~\ref{thm:2AP} to other even values of $d$.

\begin{conjecture}\label{con:c15}
Let $b$ be odd such that $b\neq2^r+1$ for any positive integer $r$, let $d$ be even, and let $p$ be the smallest prime such that $p\mid(b-1)$ and $p\nmid d$. Then there exist infinitely many $b$-anti-Niven $d$-APs of length $p-1$.
\end{conjecture}

Theorem~\ref{thm:smallerthanbound} established an upper bound for the maximum length of a $b$-anti-Niven $d$-AP when $b\geq 6$ is even and $3\leq d\leq b/2$ is an odd integer.  We conjecture that this bound is attainable for infinitely many pairs $(b,d)$.

\begin{conjecture}\label{con:c07}
    Let $b\geq6$ be even, and let $3\leq d\leq b/2$ be an odd integer. There exist infinitely many pairs $(b,d)$ for which there is a $b$-Niven $d$-AP of length $\lceil2b/d\rceil+2$.
\end{conjecture}

\section{Acknowledgements}

These results are based on work supported by the National Science Foundation under grant numbered MPS-2150299.

\end{document}